\documentclass[11pt]{amsart}
\usepackage{amsmath,amssymb,graphicx,latexsym,mathrsfs,hyperref,verbatim}
\usepackage[all]{xy}
\usepackage[bbgreekl]{mathbbol}
\newtheorem{thm}{Theorem}[section]
\newtheorem{cor}[thm]{Corollary}

\newtheorem{prop}[thm]{Proposition}
\theoremstyle{definition}
\newtheorem{defn}[thm]{Definition}
\newtheorem{example}[thm]{Example}
\theoremstyle{remark}
\newtheorem{rem}[thm]{Remark}

\numberwithin{equation}{section}
\def\Cb{{\mathbb C}}\def\Rb{{\mathbb R}}\def\Tb{{\mathbb T}}\def\Zb{{\mathbb Z}}

\def\hf{{\mathfrak h}}
\def\Cc{{\mathcal C}}\def\Tc{{\mathcal T}}

\DeclareMathOperator{\Aut}{Aut}
\DeclareMathOperator{\End}{End}
\DeclareMathOperator{\Hom}{Hom}
\DeclareMathOperator{\Hilb}{Hilb}
\DeclareMathOperator{\Vect}{Vect}

\DeclareMathOperator{\Reps}{Reps}
\DeclareMathOperator{\Flat}{Flat}

\DeclareMathOperator{\Irreps}{Irreps}

\DeclareMathOperator{\Arclength}{Arclength}
\DeclareMathOperator{\Objects}{Objects}

\newcommand{\isom}[1]{\stackrel{#1}{\hookrightarrow}}
\newcommand{\tof}[1]{\stackrel{#1}{\longrightarrow}}
\newcommand{\Tof}[1]{\stackrel{#1}{\rightarrow}}

\newcommand{\dreps}{{d^{\Reps}}}
\newcommand{\dbar}{\overline{d}}

\newcommand{\Lreps}{{\ell^{\Reps}}}
\newcommand{\wh}{\widehat}
\newcommand{\Ldd}{{\ell^{dd}}}
\newcommand{\Uh}{{U\!(\hf\!)}}
\newcommand{\LUh}{\ell^{\Uh}}

\newcommand{\dFlat}{d^{\Flat}}
\begin{document}
\title{Metric Tannakian duality}
\author{Calder Daenzer}
\thanks{This work was supported by the Alexander von Humboldt Foundation and NSF grant DMS-0703718.}
\address{Institute of Mathematics, Bunsenstrasse 3-5, G\"{o}ttingen, 37073 Germany}%
\email{calder.daenzer@gmail.com}
\maketitle
\begin{abstract}  We incorporate metric data into the framework of Tannaka-Krein duality.  Thus, for any group with left invariant metric, we produce a dual metric on its category of unitary representations.  We characterize the conditions under which a``double-dual" metric on the group may be recovered from the metric on representations, and provide conditions under which a metric agrees with its double-dual.  We also consider some applications to T-duality and quantum Gromov-Hausdorff distance.
\end{abstract}
\tableofcontents
\section{Introduction}
In this note we show how a left invariant metric on a group induces a notion of distance between its representations, and conversely, how a distance function on a group's representations induces a (semi-)metric on the group.   Applying the two methods in succession to a metric group produces a new ``double-dual" (semi-)metric on the group; we characterize the conditions under which the original metric agrees with its double-dual.  We also show that this procedure of producing new metrics stabilizes, in the sense that the double-dual and the double-double-dual metric are equal.

In the case of compact topological groups, these constructions can be viewed as the incorporation of metric data into the Tannaka duality \cite{Ta}:
\[ G\longleftrightarrow \Reps(G) \]
between a group and its category of finite dimensional complex representations.

Alternatively, one might view the constructions as a group theoretic analogue of the duality introduced by Connes \cite{Co}:
\[ (X,d)\longleftrightarrow (C(X),L_d) \]
between a locally compact metric space $(X,d)$ and the commutative $C^*$-algebra with Lip-norm $(C(X),L_d)$.

In our view the theory is interesting because of two types of applications.  First, there exist subtle relationships between groups which can only be expressed as equivalences of categories associated to them (we have in mind T-duality, Fourier-Mukai equivalence, and some cases of mirror symmetry), so if metric data can be encoded in the relevant categories, it can be transported across the equivalences.  The second (and more speculative) application is that Rieffel's quantum Gromov-Hausdorff distance \cite{Rie1} can be used to define a metric on the representations of a compact Lie group.  Comparing this with our type of metric on representations, one might extract new information about quantum-Gromov-Hausdorff convergence.  We will consider some of these applications in Section \eqref{S:Examples}.

The outline of the note is as follows: In Section \eqref{S:dreps} we define the distance between two representations of a metric group.  This distance depends on a certain choice, and we examine two such choices--- one which is best suited for finite dimensional Lie group representations, and another which applies for arbitrary representations.  In Section \eqref{S:Ldd} we define the double-dual metric.  We then generalize this in the spirit of Tannaka-Krein duality: for a category $\Tc$ with a metric on its objects and a functor to Hilbert spaces $\Tc\tof{F}\Hilb$, we produce a (possibly infinite) semi-metric on the group $\Aut F$ of natural equivalences of $F$.  The last section is devoted to examples and applications.

\underline{Acknowledgements} I would like to thank Peter Dalakov, Thomas Schick, and Marc Rieffel for the many useful interactions they provided.
\section{The metric on Representations}\label{S:dreps}

Let $G$ be a group with \textbf{semi-length} function $\ell$.  Thus $\ell$ is a function $G\to [0,\infty]$ satisfying $\ell(id_G)=0$, $\ell(g)=\ell(g^{-1})$, and $\ell(gg')\leq\ell(g)+\ell(g')$.  Semi-length functions correspond to left invariant semi-metrics on $G$ via
\[ \ell(g)= d(id_G,g)\ \longleftrightarrow\ d(g,g')=\ell(g^{-1}g'). \]
If a semi-length function vanishes only at the identity then it is called a \textbf{length} function; length functions correspond to left invariant metrics.  By default our semi-lengths and semi-metrics are allowed to take infinite values, but if they do not we refer to them as \textbf{finite}.

Now let us define the distance $\dreps(\rho,\sigma)$ between two representations of a group with semi-length function.  The result will depend on a choice, which we make now, of an infinite dimensional  Hilbert space $\hf$ and a bi-invariant metric $d^{\Uh}$ on its group $\Uh$ of unitary operators.  Two such choices are examined at the end of the section.

First, given two representations $\rho,\sigma:G\to\Uh$, form the distance:
\[ \dbar(\rho,\sigma):=\sup_{\ell(g)\neq 0}\frac{d^\Uh(\rho_g,\sigma_g)}{\ell(g)}. \]
This function is symmetric and bi-invariant because $d^\Uh$ is so, and it satisfies the triangle inequality as follows immediately from:
\[ \sup_{\ell(g)\neq 0}\frac{d^\Uh(\rho_g,\tau_g)+d^\Uh(\tau_g,\sigma_g)}{\ell(g)}
 \leq\sup_{\ell(g)\neq 0}\frac{d^\Uh(\rho_g,\tau_g)}{\ell(g)}+\sup_{\ell(g')\neq 0}\frac{d^\Uh(\tau_{g'},\sigma_{g'})}{\ell(g')}. \]

But this is not yet the desired semi-metric for representations.  We would like to enforce that unitarily equivalent representations have zero distance from each other.  Also, we need a notion of distance between two representations whose underlying Hilbert spaces are not equal.  Thus, for representations $\rho:G\to U(V_\rho)$ and $\sigma:G\to U(V_\sigma)$, we set
\begin{equation}\label{E:dreps} \dreps(\rho,\sigma):= \inf_{V_\rho\isom{a}\hf,\ V_\sigma\isom{b}\hf}\dbar(a_*\rho,b_*\sigma). \end{equation}
Here $a,b$ label the (not necessarily surjective) isometries into $\hf$, and  $a_*\rho$ denotes the representation on $\hf$ obtained by ``adding 1's on the diagonal.'' (More precisely, decompose $\hf$ as $\hf=a(V_\rho)\oplus a(V_\rho)^\perp$, then $a_*(\rho)|_{a(V_\rho)}:=a\rho a^{-1}$ and $a_*(\rho)|_{a(V_\rho)^\perp}:=id_{a(V_\rho)^\perp}$.)

But once two isometries $V_\rho\isom{a}\hf$ and $V_\sigma\isom{b}\hf$ have been chosen, all other isometries can be obtained by composing with elements of $U(\hf)$.  Thus it is clear that
\begin{equation}\label{E:dreps2}  \dreps(\rho,\sigma)\equiv\inf_{u,v\in\Uh}\dbar(u(a_*\rho) u^*,v(b_*\sigma) v^*). \end{equation}
Of course this does not depend on the initial choice of $a$ and $b$, so they will usually be omitted.  Also, because $\dbar$ is bi-invariant gives it is only necessary to take the infimum in one variable.  Thus
\begin{equation}\label{E:dreps3}  \dreps(\rho,\sigma)\equiv\inf_{u\in\Uh}\dbar(a_*\rho,u(b_*\sigma) u^*). \end{equation}

Note that if the Hilbert space $V_\rho$ of some representation has dimension greater than $dim(\hf)$ so that there are no isometries $V_\rho\to \hf$, then Equation \eqref{E:dreps} is an infimum over the empty set which equals $+\infty$~\footnote{Infimums and supremums in our context are computed on subsets of $[0,+\infty]$, thus the infimum of the empty set is $+\infty$, while the supremum of the empty set is $0$.}.
So representations of dimension greater than $\dim(\hf)$ have infinite distance from all other representations.  They do not contribute any information for the purposes of metric duality, and we will tacitly exclude them.

With this in mind, we fix the following notation: $\Reps(G)$ denotes the category whose objects are unitary representations of $G$ on Hilbert spaces and whose arrows are G-module maps (arrows are automatically isometric).  $\Reps_{fin}(G)$ denotes the subcategory of finite dimensional representations.
\begin{defn}
A \textbf{semi-metric} on a category $\Cc$ is a semi-metric on $\Objects(C)$ for which isomorphic objects have zero distance from each other. We always assume that the objects form a set.
\end{defn}
\begin{rem} It might seem superfluous to mention categories at all since what we have defined is the same data as a metric on the set of isomorphism classes of objects.  Nonetheless, we find the terminology ``metric on a category'' worthwhile because it is functors which are used to transport metric data from one category to another.
\end{rem}
\begin{prop} The function $\dreps$ is a semi-metric on $\Reps(G)$.\end{prop}
\begin{proof}
It is clear that $\dreps$ is symmetric and vanishes on unitarily equivalent representations.  Now we verify the triangle inequality using Formula \eqref{E:dreps3} and assuming (without loss of generality) that all representations are on the fixed space $\hf$.

Given representations $\rho,\tau,\sigma:G\to \Uh$ we have
\begin{align*}
\dreps(\rho,\tau)=\inf_{u\in\Uh}\dbar(\rho,u\sigma u^*)  &\leq \inf_{u,v\in\Uh}\dbar(\rho ,u\tau u^*)+\dbar(u\tau u^*,v\sigma v^*) \\
											&=\inf_{u,u^*v\in\Uh}\dbar(\rho ,u\tau u^*)+\dbar(\tau ,u^*v\sigma v^*u) \\
											&=\dreps(\rho,\tau)+\dreps(\tau,\sigma).
\end{align*}
\end{proof}

A homomorphism $\phi:(G,\ell_G)\to(H,\ell_H)$ between groups with semi-length function is called \textbf{nonexpanding} if $\ell_G(g)\geq\ell_H(\phi(g))$ for all $g\in G$.  As long as $\ell_H$ vanishes only at the identity, the assignment $(G,\ell)\mapsto (\Reps(G),\dreps)$ preserves nonexpansion:
\begin{prop}Suppose that $(G,\ell_G)\tof{\phi} (H,\ell_H)$ is a nonexpanding homomorphism of semi-length groups and $\ell_H$ is a length function. Then the pullback functor
\[ \phi^*:\Reps H\to \Reps G,\qquad (\rho:H\to U(V))\mapsto (\rho\circ\phi:G\to U(V)) \]
is also nonexpanding.
\end{prop}
\begin{proof}  For representations $\rho:H\to U(V_\rho)$ and $\sigma:H\to U(V_\sigma)$, we have
\begin{align*}  \dbar(\phi^*(\rho),\phi^*(\sigma))&=\sup_{\ell_G(g)\neq 0}\frac{d^\Uh(\rho_{\phi_g},\sigma_{\phi_g})}{\ell_G(g)} \\
                                                  &=\sup_{\ell_G(g)\neq 0,\phi g\neq id_H}\frac{d^\Uh(\rho_{\phi g},\sigma_{\phi g})}{\ell_G(g)} \\
                                                  &\leq \sup_{\phi g\neq id_H}\frac{d^\Uh(\rho_{\phi(g)},\sigma_{\phi(g)})}{\ell_H(\phi g)} \\
                                                  &=\sup_{h\in\phi(G),\ell_H(h)\neq 0}\frac{d^\Uh(\rho_h,\sigma_h)}{\ell_H(h)}\leq \dbar(\rho,\sigma).
\end{align*}
This inequality $\dbar(\phi^*(\rho),\phi^*(\sigma))\leq\dbar(\rho,\sigma)$ persists through the infimum \eqref{E:dreps2} which defines $\dreps$, thus we get $d^{\Reps(G)}(\phi^*(\rho),\phi^*(\sigma))\leq d^{\Reps(H)}(\rho,\sigma)$ as desired.
\end{proof}

Recall that the assignment $(G,\ell)\longmapsto(\Reps(G),\dreps)$ still depends on the choice of bi-invariant metric on $\Uh$.  We finish the section by writing down two natural choices for such a metric:
\begin{enumerate}
\item The operator norm $||a||:=\sup_{x\in\hf, ||x||=1}||a(x)||$ for bounded operators on $\hf$ induces a bi-invariant metric, the \textbf{operator metric}:
\[ d^{op}(u,v):=||u-v||,\qquad u,v\in U(\hf).\]
\item Recall that the Hilbert-Schmidt operators $HS(\hf)$ are the bounded operators on a separable Hilbert space $\hf$ such that $tr(A^*A)<\infty$.  For a differentiable path $[0,1]\tof{\gamma}\Uh$ whose derivative $\partial_t\gamma(t)$ is in $HS(\hf)$ for $t\in [0,1]$, there is a notion of arclength:
\begin{equation}\label{E:Length} \Arclength(\gamma):=\int_0^1 \sqrt{tr((\partial_t\gamma)^*(\partial_t\gamma))}dt.
\end{equation}
This induces a bi-invariant metric, the \textbf{arclength metric}, on $U(\hf)$ by the formula:
\[ d^{arc}(u,v):=\inf \{ \Arclength(\gamma) \ |\ \gamma(0)=u, \gamma(1)=v,\ \partial_t\gamma\in HS(\hf)\}. \]
\end{enumerate}
There are plenty of points $u,v\in\Uh$ which cannot be connected by a path $\gamma$ whose derivative is Hilbert-Schmidt, for example $id_\Uh$ and $e^{i\theta}id_\Uh$ for $\theta\in(0,2\pi)$ cannot be so connected.  In such cases $d^{arc}(u,v)$ reduces to an infimum over the empty set, which is $\infty$.  This may seem like a significant drawback, but of course for finite dimensional representations the problem does not occur, and in many cases the benefit of having a Riemannian metric outweigh the benefits of having a finite metric.

\section{Metric Duality}\label{S:Ldd}

As before, let $(G,\ell)$ denote a group with semi-length function and $\dreps$ the associated semi-metric on $\Reps(G)$.  Associated to this data is a length-type function on $\Reps(G)$
\[ \Lreps(\rho):=\dreps(1,\rho) \]
where $1$ denotes the trivial representation of $G$.  From the bi-invariance of $d^\Uh$ this formula reduces to
\begin{equation}\label{E:Lreps}   \Lreps(\rho)=\dbar(1,\rho)=\sup_{\ell(g)\neq 0}\frac{\LUh(\rho_g)}{\ell(g)}. \end{equation}

Of course $\Objects(\Reps(G))$ is not a group, so $\Lreps$ cannot be a length function in the sense of groups.  However, it is a monoid with respect to tensor product of representations and has an involution (sending a representation $\rho$ to its dual, or conjugate, representation $\bar{\rho}$).  With respect to these structures $\Lreps$ behaves analogously to a length function for groups:
\begin{prop}\label{P:PropertiesOfLreps} Let $(G,\ell)$ be a group with semi-length function and choose the arclength metric on $\Uh$ to form $\Lreps$.  Then for $\rho,\sigma\in\Reps(G)$
\begin{enumerate}
\item $\Lreps(\rho\otimes\sigma)\leq\Lreps(\rho)+\Lreps(\sigma).$
\item $\Lreps(\rho)=\Lreps(\bar{\rho}).$
\item If $\ell$ is finite and vanishes only at the identity, then $\Lreps$ vanishes only at the trivial representation.
\end{enumerate}
\end{prop}
\begin{proof} The second and third statements are immediate from Formula \eqref{E:Lreps}, and incidentally hold for any choice of bi-invariant metric on $\Uh$.  The first statement is true if
\[ \ell^\Uh(\rho_g\otimes\sigma_g)\leq\ell^\Uh(\rho_g)+\ell^\Uh(\sigma_g) \]
holds for all $g\in G$.
And this does hold, because if $\gamma_\rho$ and $\gamma_\sigma$ are paths in $U(V_\rho)$ and $U(V_\sigma)$ respectively, then $\gamma_\rho\otimes\gamma_\sigma:[0,1]\to U(V_\rho\otimes V_\sigma)$ satisfies
\[ \Arclength(\gamma_\rho\otimes\gamma_\sigma)\leq \Arclength(\gamma_\rho)+\Arclength(\gamma_\sigma). \]
Indeed, arclength is computed by integrating the norm $||A||=\sqrt{tr(A^*A)}$ where $A$ is the derivative of the path, and since $\partial_t(\gamma\otimes\eta)=\partial_t\gamma\otimes\eta+\gamma\otimes\partial_t\eta$ we have
\[ ||\partial_t(\gamma\otimes\eta)||  \leq ||\partial_t\gamma\otimes\eta||+||\gamma\otimes\partial_t\eta||=||\partial_t\gamma||+||\partial_t\eta||. \]
(The equality on the right is because $\gamma$ and $\eta$ are unitary.)
\end{proof}
\begin{example}[Pontryagin duality] Let $(G,\ell_G)$ be a locally compact abelian group with finite length function and let $\wh{G}:=\Hom(G,U(1))$ denote the Pontryagin dual.  Define $\Lreps$ using the arclength metric.  If we view $\wh{G}$ as a subset of $\Reps(G)$ then multiplication in $\wh{G}$ corresponds to tensor product in $\Reps(G)$, so Proposition \eqref{P:PropertiesOfLreps} implies that $\wh{\ell}:=\Lreps|_{\wh{G}}$ is a finite length function (in the sense of groups) on $\wh{G}$.  We refer to $\wh{\ell}$ as the \textbf{Pontryagin dual length function}.
\end{example}
The reader can check directly that for $G=\Zb$ with length function $\ell(k):=\alpha|k|$ for $\alpha\in(0,\infty)$, the Pontryagin dual length function on $\wh{\Zb}=U(1)$ makes $U(1)$ the round circle of radius $1/\alpha$.   The following proposition gives an analogous formula for the groups $\Rb^n$ and $\Tb^n$, when they are equipped with left invariant Riemannian metrics:
\begin{prop}\label{P:PontLength}\ \newline
\begin{enumerate}
\item Let $G=\Rb^n$, equipped with the length function $\ell(x):=\sqrt{\langle x, Ax\rangle}$, where $A$ is a positive definite symmetric $n\times n$ matrix and $\langle -, -\rangle$ is the standard inner product.  Then the Pontryagin dual length function on $\wh{\Rb^n}$ is:
\begin{equation}\label{E:ellhat}  \wh{\ell}(\phi_p)=\sqrt{\langle p,A^{-1}p\rangle},\qquad\phi_p(x):=e^{i\langle x,p\rangle},\ p,x\in\Rb^n. \end{equation}
\item Let $\Tb^n=\Rb^n/\Lambda$ (here $\Lambda\simeq \Zb^n$ is a lattice), and view $\wh{\Tb^n}$ as a lattice in $\wh{\Rb^n}$.  Let $\ell$ denote the length function from (1) and write $q_*\ell$ for the quotient length function:
\[ q_*\ell(x+\Lambda):=\inf_{\lambda\in\Lambda}\ell(x+\lambda)\qquad x+\Lambda\in \Tb^n.\]
Then the Pontryagin dual of $q_*\ell$ is $\wh{\ell}$ restricted to $\wh{\Tb^n}$.  That is,
\[ \wh{q_*\ell}(\phi_p)=\sqrt{\langle p,A^{-1}p\rangle},\qquad\phi_p\in\wh{\Tb^n}\subset\wh{\Rb^n}. \]
\end{enumerate}
\end{prop}
\begin{proof}
We start with the first statement.  Rewriting $y=A^{1/2}x\in\Rb^n$ in the definition of $\Lreps$ gives
\[  \wh{\ell}(\phi_p)=\sup_x \frac{\ell^{arc}(e^{i\langle x,p\rangle})}{\sqrt{\langle x,Ax\rangle}}
     =\sup_y\frac{\ell^{arc}(e^{i\langle y,A^{-1/2}p\rangle})}{||y||} \]
This supremum is clearly achieved on the line segment $\{ y=\lambda A^{-1/2}p\ |\ 0<|\lambda|< 2\pi/||A^{-1/2}p||^2\ \}$, and the supremum is
\[ \frac{\ell^{arc}(e^{i\lambda\langle A^{-1/2}p,A^{-1/2}p\rangle})}{|\lambda|\cdot||A^{-1/2}p||}=\sqrt{\langle p,A^{-1}p\rangle} \]
as desired.

Now we show that $\wh{q_*\ell}(\phi_k)=\wh{\ell}(\phi_k)$ for any point $\phi_k\in \wh{\Tb^n}\subset\wh{\Rb^n}$.  Such a point is defined by $\phi_k(x+\Lambda)=e^{i\langle x,k\rangle}$ where $k$ satisfies $\langle \Lambda,k\rangle\subset 2\pi\Zb$.   Then
\[ \wh{q_*\ell}(\phi_k)\equiv\sup_{x+\Lambda\in \Tb^n}\ \frac{\ell^{arc}(e^{i\langle x,k\rangle})}{\inf_{\lambda\in\Lambda}\ell(x+\lambda)}
=\sup_{x}\sup_\lambda\frac{ \ell^{arc}(e^{i\langle x+\lambda,k\rangle})}{\ell(x+\lambda)}=\sup_{y\in\Rb^n}\frac{\ell^{arc}(e^{i\langle y,k\rangle})}{\ell(y)}\equiv\wh{\ell}(\phi_k). \]
(We substituted $y=x+\lambda$ and used the fact that $e^{i\langle x,k\rangle}=e^{i\langle x+\lambda,k\rangle}$.)
\end{proof}
An immediate consequence of Equation \eqref{E:ellhat} is that those length functions on $\Rb^n$ satisfy \[ \wh{\wh{\ell}}=\ell. \] In particular, one can recover the original length function on $\Rb^n$ from knowledge of $\Lreps$.

This leads to the general question which we will now address: from which semi-length groups can one recover the semi-length function from $\Lreps$?

\begin{defn} Let $(G,\ell)$ be a group with semi-length function and $\Lreps$ a corresponding length on $\Reps(G)$.  The \textbf{double-dual} of $\ell$, written $\Ldd$, is the function on $G$ given by:
\[  \Ldd(g):=\sup_{\Lreps(\rho)\neq 0}\frac{\LUh(\rho_g)}{\Lreps(\rho)}\in [0,\infty]. \]
\end{defn}
In fact $\Ldd$ is a semi-length function, as follows immediately from the properties of $\LUh$.
\begin{prop}\label{P:Ldd} Let $(G,\ell)$ be a group with length function.  Then:
\begin{enumerate}
\item $\Ldd(g)\leq\ell(g)$ for all $g\in G$.
\item If $G$ admits an isometric representation then $\Ldd=\ell$.
\item More generally, $\Ldd=\ell$ if for each $g\in G$ there exists a representation $\rho$ such that
\[ \Ldd(g)=\frac{\ell^{U(\hf)}(\rho_g)}{\Lreps(\rho)}. \]

\end{enumerate}
\end{prop}
\begin{proof}
In Statement (1), if $\Ldd(g)=0$ there is nothing to prove.  On the other hand, if $\Ldd(g)\neq 0$ then by definition there is some representation $\rho$ of finite length with $\rho(g)\neq 0$.  This ensures that there are no zeroes in denominators in following equations which show that $\Ldd(g)\leq\ell(g)$:
\begin{align*} \Ldd(g)&:=\sup_{\Lreps(\rho)\neq 0}\frac{\LUh(\rho_g)}{\Lreps(\rho)} \\
                             &=\sup_{\Lreps(\rho)\neq 0}\left(\frac{\LUh(\rho_g)}{ \left( \sup_{h\neq id}\frac{\LUh(\rho_h)}{\ell(h)}\right)}\right) \\
				&\leq\sup_{\Lreps(\rho)\neq 0}\frac{\LUh(\rho_g)}{\left( \frac{\LUh(\rho_g)}{\ell(g)}\right)}=\ell(g).
\end{align*}
Now Statement (2) follows easily, indeed, any representation satisfying $\ell(g)=\LUh(\rho_g)$ for all $g\in G$ satisfies $\Lreps(\rho)=1$, so that
\[ \Ldd(g)\equiv\sup_{\Lreps(\sigma)\neq 0}\frac{\LUh(\sigma_g)}{\Lreps(\sigma)}\geq \frac{\LUh(\rho_g)}{\Lreps(\rho)}=\ell(g) \]
which forces the equality $\Ldd=\ell$.  Statement (3) is immediate from the definitions.
\end{proof}

\begin{defn} A semi-length function $\ell$ on a group $G$ will be called \textbf{stable} if $\Ldd=\ell$.  The notion of stability depends on the choice of bi-invariant metric $d^\Uh$ on $\Uh$; when we wish to make this explicit we will say $d^\Uh$-stable. \end{defn}

\begin{prop}\label{P:CompactStability}\
\begin{enumerate}
\item Every subgroup of the unitary group $\Uh$ admits a finite stable length function.
\item The translation invariant Riemannian metrics on compact Lie groups are stable.
\end{enumerate}
\end{prop}
\begin{proof} For the first statement, let $G$ be a subgroup of $U(\hf)$ and define $\ell_G:=\LUh|_G$, where $\LUh$ is the operator length.  Then the inclusion $G\hookrightarrow\Uh$ is an isometric representation.  Thus $\ell_G$ is stable, and it is finite because the operator length is so.

Now let $(G,\ell)$ be a compact Lie group with Riemannian metric.  Then the second statement follows because there is an isometry $G\hookrightarrow U(\hf)$, where $U(\hf)$ is equipped with the arclength metric.
\end{proof}

But many groups which are not subgroups of $\Uh$, and in particular non-compact groups, may also have stable length functions. The integers provide a simple example of this:
\begin{example}  The group $\Zb$ with the length function satisfying $\ell(n)=\ell(1)\cdot|n|$ satisfies $\Ldd=\ell$.  Indeed,
\[ \frac{\LUh(\rho_n)}{\ell(1)|n|}\leq \frac{\LUh(\rho_1)|n|}{\ell(1)|n|}=\frac{\LUh(\rho_1)}{\ell(1)},\]
so that
\[ \Lreps(\rho)\equiv\frac{\ell(\rho_1)}{\ell(1)}\text{ and thus }\Ldd(n)=\sup_\rho\frac{\ell(\rho_n)\ell(1)}{\ell(\rho_1)}.   \]
But the character $\rho(k)=e^{ikx}$ for any $x\in(0,\pi/|n|)$ satisfies $\ell(\rho_n)=\ell(\rho_1)\cdot|n|$ for the arclength metric and realizes the supremum, thus $\Ldd(n)=\ell(n)$.
\end{example}
In fact since the supremum is achieved with an irreducible representation, we have shown the stronger statement that $\wh{\wh{\ell}}(n)=\ell(n)$.
It is stronger in the following sense:
\begin{prop}\label{P:LhathatLdd}
Let $(G,\ell)$ be a locally compact abelian group with finite length function. Then \[ \wh{\wh{\ell}}\leq\Ldd\leq\ell.\]
\end{prop}
\begin{proof} The inequality $\wh{\wh{\ell}}\leq\Ldd$ is immediate from the definitions, and $\Ldd\leq \ell$ is Proposition \eqref{P:Ldd}.\end{proof}
While it is not easy in general to check whether or not a given metric is stable, there is a canonical method for producing stable metrics:
\begin{prop} Let $(G,\ell)$ be a group with finite length function.  Then on the set
\[ \Reps_c(G):=\{\rho\in \Reps(G)\ |\ \Lreps(\rho)\neq \infty\} \]
we have $(\Ldd)^{\Reps}=\Lreps$.  Consequently the double-dual is always a stable semi-length function.
\end{prop}
\begin{proof}  First we show that $\Lreps\leq(\Ldd)^{\Reps}$.  Because $\ell$ is finite, $\Lreps(\rho)=0$ only when $\rho$ is the trivial representation, and in this case $(\Ldd)^{\Reps}(\rho)=0$ as well.  Now consider the case that $\Lreps(\rho)\neq 0,\infty$.
We claim that:
\begin{equation}\label{E:prop} \Lreps(\rho)\equiv\sup_{\ell(h)\neq 0}\frac{\ell(\rho_h)}{\ell(h)}= \sup_{\Ldd(h)\neq 0}\frac{\ell(\rho_h)}{\ell(h)}. \end{equation}
This is because $\Ldd(g)=0$ only if $\ell(\rho_g)$ vanishes for every $\rho\in\Reps_c(G)$, and in this case $g$ clearly does not contribute to the supremum.  By Proposition \eqref{P:Ldd}, $\Ldd(g)\leq\ell(g)$ for all $g$, so Equation \eqref{E:prop} implies that  $\Lreps(\rho)\leq(\Ldd)^{\Reps}(\rho)$ as desired.

To get $(\Ldd)^{\Reps}\leq\Lreps$ for $\Lreps(\rho)\neq 0,\infty$ we expand:
\[ (\Ldd)^{\Reps}(\rho)=\sup_{\Ldd(g)\neq 0}\frac{\ell(\rho_g)}{\Ldd(g)}
    =\sup_{\Ldd(g)\neq 0}\frac{\ell(\rho_g)}{\sup_{ \Lreps(\sigma)\neq 0,\infty}\frac{\ell(\sigma_g)}{\Lreps(\sigma)}}
    \leq\Lreps(\rho). \]
The last inequality is given by fixing $\sigma=\rho$.
\end{proof}

In summary, from the data of a semi-length group $(G,\ell)$, we produce a semi metric $\dreps$ and associated length-type function $\Lreps$ on $\Reps(G)$.  The function $\Lreps$ is analogous to a length function in the sense that if satisfies a triangle inequality with respect to tensor product of representations, and assigns the same value to a representation and its dual.  From $\Lreps$ one can produce a new semi-length function $\Ldd$ on $G$ which is dominated by $\ell$.  If $\ell$ is actually a length function, then $\Ldd$ is stable (that is, $(\Ldd)^{dd}=\Ldd$) because $\Lreps=(\Ldd)^{\Reps}$.

It is pleasing to frame these methods in the context of Tannaka-Krein duality:  Let $\Hilb$ denote the category of  Hilbert spaces and isometries.  If a category over Hilbert spaces $\Tc\Tof{F}\Hilb$ is equipped with a function $\ell^\Tc:Ob\Tc\to[0,\infty]$, then we obtain a semi-length function on the group of natural isomorphisms of $F$ by the formula:
\[ \ell^F:\Aut(F)\to [0,\infty],\qquad \ell^F(\eta):=\sup_{\ell^\Tc(c)\neq 0}\frac{\ell^{U(F(c))}(\eta(c))}{\ell^\Tc(c)}. \]
(As always, $\ell^{U(F(c))}$ is determined by a choice as in Section \eqref{S:dreps}.)

Now if $\Tc:=\Reps_{fin}(G)$ is the (monoidal) category of continuous finite dimensional unitary representations of a compact group $G$, then one has a topology on $\Aut(F)$ and a canonical realization of $G$ as a subspace of $\Aut(F)$.  In this case $\ell^F|_G$ is a semi-length function on $G$.  Furthermore, if $\ell^\Tc$ is the semi-length function on representations induced from a semi-length $\ell$ on $G$ (that is $\ell^\Tc=\Lreps$), then by definition $\ell^F|_G\equiv \Ldd$.  In particular, the results of this section imply:
\begin{cor} Let $(G,\ell)$ be a compact group with stable length function.  Then Tannaka-Krein duality incorporates metric data via the associations
\[ (G,\ell)\rightsquigarrow (\Tc,\ell^\Tc):=( \Reps_{fin}(G),\Lreps), \]
\[ (\Tc\Tof{F}\Hilb,\ell^\Tc)\rightsquigarrow (G,\ell^F|_G\equiv\ell). \]
\end{cor}
\section{Examples}\label{S:Examples}
To conclude this note we would like give some rather open-ended examples, whose unifying theme is that of metric data from a group encoded in various associated categories.

\subsection{Metrics on flat bundles via the Riemann-Hilbert correspondence} Let $(M,G)$ be a Riemannian manifold and let $\widetilde{M}\tof{p}M$ be the universal cover, with its lifted Riemannian metric.  Choosing an identification $\pi_1(M,x)\simeq p^{-1}(x)$ for some $x\in M$ induces a translation invariant metric on $\pi_1(M,x)$.  As we have seen, this in turn determines a semi-metric $\dreps$ on $\Reps(\pi_1(M,x))$.

The Riemann-Hilbert correspondence provides an equivalence between the unitary representations $\Reps_{fin}(\pi_1(M,x))$ and the category $\Flat(M)$ of Hermitian vector bundles on $M$ with flat unitary connection:
\[  Hol:\Flat(M)\leftrightarrows\Reps_{fin}(\pi_1(M,x)):Assoc. \]
The equivalence sends a bundle with connection to its holonomy representation at $x$, and in the other direction, sends a representation $(\rho:\pi\to U(V))$ to the associated bundle $\tilde{M}\times_\rho V$, which admits a canonical flat unitary connection.)

Transporting $\dreps$ across the Riemann-Hilbert correspondence gives a metric $\dFlat$ on the objects of $Flat(M)$, so we have a map:
\[ \{\text{ Riemannian metrics on } M\ \}\longrightarrow \{\text{ Semi-metrics on }\Flat(M)\ \}. \]
When $M$ is a compact Riemann surface of genus $g>1$, $\Flat(M)$ is precisely the set of stable degree 0 bundles on $M$.  Its moduli space admits a K\"{a}hler metric (see \cite{NaSe}) which might be interesting to compare with $\dFlat$.

\subsection{Metric T-duality}
Consider a torus $T:=\Rb^n/\Lambda$.  The associated \textbf{T-dual torus}, denoted $T^\vee$, is by definition the Pontryagin dual group of $\Lambda$.  In view of the inclusion
 \[T^\vee\equiv\wh{\Lambda}\hookrightarrow \Reps(\pi_1(T))\tof{\sim} \Flat(T),\]
the construction from the previous paragraph yields a map
\[ \{\text{ Riemannian metrics on } T\ \}\longrightarrow \{\text{ Riemannian metrics on } T^\vee\ \}\quad d\mapsto \dFlat|_{T^\vee}. \]
From the perspective of T-duality one would like this be a bijection. Unfortunately this cannot be, because on $T$ we start with an arbitrary Riemannian metric, whereas the resulting metric on $T^\vee$ is by construction translation invariant.

In any case, the metric aspect of T-duality only applies to translation invariant metrics.  If we restrict to these from the beginning then there are several routes by which one can produce a metric on the dual torus.  For instance, letting $d_A$ denote the metric specified by a positive definite symmetric bi-linear form $A$ on $\Rb^n$, we have an isomorphism:
\[ \{\text{ translation inv. metrics on } T\ \}\simeq\{\text{ translation inv. metrics on } T^\vee\ \}\quad d_A\mapsto d_{A^{-1}}. \]
This is the correct map for T-duality, as dictated by the Buscher rules \cite{Bu}.   On the other hand, $d_A\mapsto d_{A^{-1}}$ can be interpreted as a metric duality because $d_A$ and $d_{A^{-1}}$ are Pontryagin dual in the sense of Proposition \eqref{P:PontLength}, once they are lifted to the respective covers $\Rb^n$ and $\wh{\Rb^n}$.

\begin{rem} In the more general context of $T$-duality for torus bundles, metric data can induce so-called $B$-fields on the dual side.  It should be possible to include this feature in our framework by encoding metric and B-field data as a ``generalized Riemannian structure" (defined in \cite{CaGu}).
\end{rem}
\subsection{Metrics via the Peter-Weyl correspondence}
Let $G$ be a compact Lie group and $\wh{G}$ the discrete space of isomorphism classes of irreducible representations of $G$.  We denote by $\Vect(\wh{G})$ the category whose objects are vector bundles over $\wh{G}$ and whose arrows are bundle isomorphisms.  Of course a vector bundle on $\wh{G}$ is just a collection of vector spaces $\{V_\phi\}_{\phi\in\wh{G}}$ of varying dimension.

There is no map between $G$ and $\wh{G}$, so there is no direct way to transport metric data from one side to the other.  However there is an equivalence between the representations of $G$ and the vector bundles on $\wh{G}$, so it is possible to transport metric data by encoding it in these categories.  A convenient way to describe the equivalence is as follows:

 Let $W=\oplus_{\phi\in\wh{G}}V_\phi$ be the vector bundle over $\wh{G}$ whose fiber over a point $\phi$ is the vector space $V_\phi$ of the representation.  It has an obvious $G$-action preserving the bundle structure.  We have a functor:
\[ \phi_W:\Reps_{fin}(G)\longmapsto\Vect(\wh{G}) \]
\[ V\longmapsto \Hom_G(V,W) \]
This is easily seen to be an equivalence; Schur's lemma implies $\Hom_G(V,W)$ is a vector bundle whose fiber over $\phi\in\wh{G}$ has dimension equal to the number of copies of $V_\phi$ occurring in $V$.

Using this equivalence, we see that semi-metrics on $G$ give rise to semi-metrics on $Vect(\wh{G})$, and vice-versa.

\subsection{Metrics on representations and Quantum Gromov-Hausdorff distance}

In our last example we present a class of metrics on the irreducible representations of a compact Lie group which we define using Rieffel's quantum Gromov-Hausdorff distance \cite{Rie1}, \cite{Rie2}.

The setup is as follows (see \cite{Rie3} for the full details):  If $G$ is a compact Lie group with length function $\ell$, then any irreducible representation $\rho:G\to U(V_\rho)$ gives rise to a quantum metric space $(A_\rho,L_\rho)$.  The underlying normed space is $A_\rho:=\End_\Cb(V_\rho)$ (equipped with, for example, the Hilbert-Schmidt norm) and the Lip-norm is given by
\[ L_\rho:A_\rho\to [0,\infty]\qquad L_\rho(a):=\sup_{g\neq id_G}\frac{||[\rho_g,a]||}{\ell(g)}. \]
Thus for any such $(G,\ell)$, we have a metric on irreducible representations given by
\[ \Irreps(G)\ni \rho,\tau\longmapsto d^{QGH}(A_\rho,A_\tau), \]
where $d^{QGH}$ stands for quantum Gromov-Hausdorff distance.

On the other hand from $(G,\ell)$ one can form the semi-metric $\dreps$.  But then it is immediately tempting to seek a relationship between $\dreps(\rho,\tau)$ and $d^{QGH}(A_\rho,A_\tau)$.
A relation of the form $\dreps\leq d^{QGH}$ (or vice versa) would provide information about convergence of matrix algebras to homogenous spaces (as in \cite{Rie3}), and this could be of great value because $\dreps$ is more easily computable than $d^{QGH}$.

But recalling that $\dreps$ has some freedom coming from its dependence on the initial choice of a metric $d^{U(\hf)}$, we could ask a stronger question:  Is it possible to make two different choices of $d^{U(\hf)}$, one leading to $\dreps\leq d^{QGH}$ and the other leading to $d^{QGH}\leq \dreps$?  We leave this as an open problem.

\end{document}